\documentclass{article}
\usepackage{latexsym}
\usepackage{amssymb}
\usepackage{amsmath}
\usepackage{color}
\usepackage{graphicx}
\usepackage{amsthm}
\usepackage{indentfirst}
\usepackage{mathrsfs}
\allowdisplaybreaks

\newtheorem{theorem}{\color{black}\indent Theorem}[section]
\newtheorem{lemma}{\color{black}\indent Lemma}[section]

\newtheorem{definition}{\color{black}\indent Definition}[section]
\newtheorem{remark}{\color{black}\indent Remark}[section]
\newtheorem{corollary}{\color{black}\indent Corollary}[section]
\newtheorem{example}{\color{black}\indent Example}[section]
\renewcommand{\baselinestretch}{1.2}

\DeclareMathOperator{\diag}{{diag}}
\DeclareMathOperator{\ind}{{ind}}
\DeclareMathOperator{\dist}{{dist}}
\DeclareMathOperator{\fix}{{Fix}}
\begin{document}
\large
\title{Global Weinstein Type Theorem on Multiple Rotating Periodic Solutions for Hamiltonian Systems \footnote{This work was supported by NSFC (No.11901080, 12071175), Project of Science and Technology Development of Jilin Province, China (No. 20190201302JC, 20200201270JC).}}
\author{{Jiamin Xing$^{a}$ \thanks{ E-mail address : xingjiamin1028@126.com} ,~ Xue Yang$^{b}$ \thanks{ E-mail address : yangxuemath@163.com} ,~ Yong Li$^{b,a}$} \thanks{ E-mail address : liyongmath@163.com}\\
{$^{a}$School of Mathematics and Statistics, and}\\
{Center for Mathematics and Interdisciplinary Sciences,}\\
{Northeast Normal University, Changchun 130024, P. R. China.}\\
 {$^{b}$College of Mathematics, Jilin University,}\\
 {Changchun, 130012, P. R. China.}\\
 }
\date{}
\maketitle
\par
{\bf Abstract.}
\par This paper concerns the existence of multiple rotating periodic solutions for $2n$ dimensional convex Hamiltonian systems.
For the symplectic orthogonal matrix $Q$, the rotating periodic solution has the form of $z(t+T)=Qz(t)$, which might be periodic, anti-periodic, subharmonic or quasi-periodic according to the structure of $Q$. It is proved that there exist at least $n$ geometrically distinct rotating periodic solutions on a given $Q$ invariant convex energy surface under a pinching condition. As a result, it is proved that if the symmetric energy surface admits a nonsymmetric periodic solution, it has infinitely many periodic orbits. In order to prove the result, we introduce a new index on rotating periodic orbits.
\section{Introduction and main result}

Consider the following Hamiltonian system
\begin{eqnarray}\label{1.1}
z'=J\nabla H(z),
\end{eqnarray}
where the Hamiltonian function $H\in C^1(\mathbf{R}^{2n},\mathbf{R})$ and $J$ is the canonical symplectic matrix, that is
\begin{equation}\label{1.3}
J=\left(
\begin{array}{ccc}
0&I_n\\
-I_n&0
\end{array}
\right),
\end{equation}
$I_n$ is the $n\times n$ identity matrix.
For a solution $z(t)$ of \eqref{1.1}, the corresponding orbit is the set $z(\mathbf{R})$ and two solutions $z(t)$
and $y(t)$ are geometrically distinct or have different orbits if $z(\mathbf{R})\neq y(\mathbf{R})$.

It has been an important topic that how many geometrically distinct periodic solutions a Hamiltonian system (\ref{1.1}) admits on a given energy surface, and this has achieved some remarkable progress.  Lyapunov \cite{ly} and Horn \cite{ho} first proved a local theorem, i.e., when convexity and nonresonance occur at the origin, there exist $n$ distinct periodic orbits of (\ref{1.1}) on an energy surface near the origin. Weinstein \cite{we1} and Moser \cite{mo} greatly improved these results by removing the nonresonance assumption.
Rabinowitz \cite{ra} and Weinstein \cite{we2} established the first global existence result. For a complete global versions of Weinstein's theorem, it is due to  Ekeland and Lasry \cite{ek}, under convexity and a pinching condition, see also Ambrosetti and Mancini\cite{am} for another proof.
In Ekeland and Lassoued's \cite{ek-lass}, Ekeland and Hofer's \cite{ek-ho}, and Szulkin's \cite{sz0}, it was proved that there exist at least $2$ periodic orbits on the compact convex energy surface.  Long and Zhu \cite{Long1} proved that the number of periodic orbits was at least $[\frac{n}{2}]+1$, where $[a]$ denotes the greatest
integer which is not greater than $a$.
The existence of at least $n$ periodic orbits was proved by Liu, Long and Zhu's \cite{LiuC} with the assumption that
the energy surface is symmetric with respect to the origin and by Wang, Hu, and Long's \cite{Wang} for $n=3$.

It is a nature question that if the Hamiltonian is invariant under some group action, how many symmetric solutions will exist on the energy surface.
Such symmetric theory on periodic orbits was well developed by Girardi \cite{gi}, van Groesen \cite{gr}, Rabinowitz \cite{ra1}, Szulkin \cite{sz}, Bartsch and Clapp \cite{ba}, Long, Zhang and Zhu \cite{Long}, Zhang \cite{Zhang} and Liu \cite{Liu}.
Let $\mathrm{Sp}(2n)$ be the symplectic group, that is
$$\mathrm{Sp}(2n)=\{Q\in \mathrm{GL}(2n,\mathbf{R}):\ Q^\top JQ=J\},$$
where $\mathrm{GL}(2n,\mathbf{R})$ is the group of $2n\times 2n$ invertible real matrices and $Q^\top$ is the transpose of $Q$.
Recently, Ekeland \cite{Ek} considered the solution with boundary condition $z(T)=Qz(0)$ of convex Hamiltonian system for $Q\in\mathrm{Sp}(2n)$. At the end of the paper, he raised the question of whether such a solution will still exist if the energy surface of the Hamiltonian function is bounded. We \cite{xing1} studied the case that $Q\in \mathrm{Sp}(2n)\cap\mathrm{O}(2n)$ near the equilibrium, where $\mathrm{O}(2n)$ is the orthogonal group on $\mathbf{R}^{2n}$.

In this paper, we want to find the solution of \eqref{1.1} with
$$z(t+T)=Qz(t)\ \ \forall t\in\mathbf{R},$$
on a fixed bounded energy surface, where $T>0$ and $Q\in \mathrm{Sp}(2n)\cap\mathrm{O}(2n)$. Such a solution is called a $(Q,T)$-rotating periodic solution or a $Q$-rotating periodic solution with rotating period $T$.
Following the structure of $Q$, it will be seen that
a $(Q,T)$-rotating periodic solution $z(t)$ is a special quasi-periodic one that can be obtained by the rotation of itself on $[0,T]$. Obviously, $z(t)$ is $T$-periodic if $Q=I_{2n}$, anti-periodic if $Q=-I_{2n}$  and subharmonic if $Q^k=I_{2n}$ for some integer $k>1$. This kind of rotating periodic solution has been concerned in recent years, for example, see \cite{chang,liu,xing} and references therein.

The main result of the paper is the following.
\begin{theorem}\label{th1}
Assume $Q\in \mathrm{Sp}(2n)\cap\mathrm{O}(2n)$, $H\in C^1(\mathbf{R}^{2n},\mathbf{R})$ is $Q$ invariant, i.e., $H(Qz)=H(z)$ $\forall z$ $\in\mathbf{R}^{2n}$, for $\beta>0$ the set $C=\{z\in\mathbf{R}^{2n}:H(z)\leq\beta\}$ is strictly convex and the
boundary $S=H^{-1}(\beta)$ satisfies  $\nabla H(z)\neq0$ for
every $z\in S$. Suppose that there exist $r>0$ and $R\in(r,\sqrt{2}r)$ such that
$$B_r(0)\subset C\subset B_R(0).$$
Then there exist at least $n$ geometrically distinct $Q$-rotating periodic solutions of \eqref{1.1} on $S$.
\end{theorem}

In \cite{HWK}, Hofer et al. proved that the number of periodic orbits on the compact convex energy surface of a $4$-dimensional Hamiltonian system is exactly $2$ or $+\infty$. It is conjectured that the number of periodic orbits on each compact convex energy surface of a $2n$ dimensional Hamiltonian
system is either $n$ or $+\infty$, see \cite{Wang}.
From the proof of Theorem \ref{th1}, we see that in the case of $Q=I_{2n}$, the conjecture holds under the pinching condition, which can actually be obtained in \cite{ek} and \cite{am}, but has been somewhat overlooked. Based on this fact and Theorem \ref{th1}, we have the following corollary.
\begin{corollary}
Assume the conditions of Theorem \ref{th1} hold, $Q^k=I_{2n}$ for some integer $k>1$ and system \eqref{1.1} has a periodic solution that is not $Q$-rotating periodic. Then system \eqref{1.1} has infinitely many periodic orbits on $S$.
\end{corollary}

The following simple example illustrates that the system may have an infinite number of periodic orbits on the energy surface, but only a finite number of them are $Q$-rotating periodic. It is also possible that there are only a finite number of periodic orbits, but there exist infinitely many $Q$-rotating periodic orbits.
\begin{example}
Consider the Hamiltonian
\begin{align}\label{e1}
H(x,y)=\sum\limits_{j=1}^n\frac{1}{2}\omega_j(x_j^2+y_j^2),
\end{align}
where $x=(x_1,\cdots,x_n)^\top$, $y=(y_1,\cdots,y_n)^\top$, $0<\omega_1\leq \omega_2\leq\cdots\leq\omega_n$. Assume $\omega_n<2\omega_1$, then the pinching condition holds.
Let
\begin{equation*}
Q=P\diag\left(\left(
\begin{array}{ccc}
\cos \theta_1&\sin \theta_1\\
-\sin\theta_1&\cos \theta_1
\end{array}
\right),\cdots,\left(
\begin{array}{ccc}
\cos \theta_n&\sin \theta_n\\
-\sin\theta_n&\cos \theta_n
\end{array}
\right)\right)P^\top,
\end{equation*}
where $0\leq\theta_j<2\pi$ for $1\leq j\leq n$ and $P\in \mathrm{O}(2n)$ such that
\begin{equation*}
P^{\top}JP=\diag\left(\left(
\begin{array}{ccc}
0&1\\
-1&0
\end{array}
\right),\cdots,\left(
\begin{array}{ccc}
0&1\\
-1&0
\end{array}
\right)\right).
\end{equation*}
Obviously, $H(Q(x,y))=H(x,y)$ for $(x,y)\in\mathbf{R}^{2n}$ and the solution $z(t)=(x(t),y(t))$ has the form
\begin{equation*}
z(t)=P\diag\left(\left(
\begin{array}{ccc}
\cos \omega_1t&\sin \omega_1t\\
-\sin\omega_1t&\cos \omega_1t
\end{array}
\right),\cdots,\left(
\begin{array}{ccc}
\cos \omega_nt&\sin \omega_nt\\
-\sin\omega_nt&\cos \omega_nt
\end{array}
\right)\right) z_0,
\end{equation*}
with $z_0\in\mathbf{R}^{2n}$.
It is easy to see that when $\frac{\omega_i}{\omega_j}\notin\mathbf{Q}$ for all $i\neq j$, there exist $n$ periodic orbits on the energy surface $H^{-1}(\beta)$ for $\beta>0$, otherwise there are infinite number of periodic orbits ; when $\frac{\omega_i}{\omega_j}\neq\frac{\theta_i+2k_i\pi}{\theta_j+2k_j\pi}$
for all $i\neq j$ and $k_i, k_j\in\mathbf{Z}$, there exist $n$ $Q$-rotating periodic orbits on the energy surface, otherwise it has infinitely many $Q$-rotating periodic orbits.
\end{example}
Based on Theorem \ref{th1} and the long-standing conjecture about periodic orbits, we think that the following result holds true.

\emph{Conjecture: On the compact convex energy surface of a $2n$-dimensional $Q$ invariant Hamiltonian system, there exist at least $n$ $Q$-rotating periodic orbits.}

The paper is organized as follows: In section 2, we introduce a new index on rotating periodic orbits, which satisfies the abstract framework of Rabinowitz's index theory \cite{ra-in}, but different from the $S^1$-index. In particular, our index also works on special quasi-periodic orbits when $Q^k\neq I_{2n} ~~\forall k\in\mathbf{N}_+$. We will give the proof of the main result via the index theory and Ambrosetti and Mancini's method in section 3.

\section{$\mathcal{Q}(s)$-index and abstract critical point theorem}
In this section, we will introduce a new index on rotating periodic orbits. First, let us recall the concept of index due to Rabinowitz \cite{ra-in}.

Suppose $\mathcal{E}$ is a Banach space with a group $\mathfrak{g}$ acting on it. Denote by
$$\fix\mathfrak{g}=\{z\in\mathcal{E}:gz=z,\ {\rm for\ all}\ g\in\mathfrak{g}\}.$$
Let
$$\mathfrak{\Gamma}=\{\Gamma\subset\mathcal{E}: \Gamma {\rm\ is\ closed},\ g(\Gamma)\subset\Gamma\ {\rm for\ all}\ g\in\mathfrak{g}\}$$
be the set of $\mathfrak{g}$ invariant closed subsets of $\mathcal{E}$.
\begin{definition}\label{def1}
An index for $(\mathcal{E},\mathfrak{g})$ is a mapping $i:\mathfrak{\Gamma}\rightarrow \mathbf{N}\cup\{\infty\}$ such that for
all $\Gamma_1,\Gamma_2\in\mathfrak{\Gamma}$,
\begin{enumerate}
\item[{\rm(i)}] Normalization: if $z\notin\fix\mathfrak{g}$, $i(\cup_{g\in\mathfrak{g}}gz)=1$;

\item[{\rm(ii)}] Mapping property: if $R: \Gamma_1\rightarrow\Gamma_2$ is continuous and equivariant which means
$Rg=gR$ for all $g\in\mathfrak{g}$, then
$$i(\Gamma_1)\leq i(\Gamma_2);$$

\item[{\rm(iii)}] Monotonicity property: If $\Gamma_1\subset\Gamma_2$, then $i(\Gamma_1)\leq i(\Gamma_2)$;

\item[{\rm(iv)}] Continuity property:  if $\Gamma_1$ is compact and $\Gamma_1\cap\fix\mathfrak{g}=\emptyset$, then
$i(\Gamma_1)<\infty$ and there exists a neighborhood $D\in\mathfrak{\Gamma}$ of $\Gamma_1$ such that
$$i(D)=i(\Gamma_1);$$

\item[{\rm(v)}]  Subadditivity: $i(\Gamma_1\cup\Gamma_2)\leq i(\Gamma_1)+i(\Gamma_2)$.
\end{enumerate}
\end{definition}

By the assumption $Q\in \mathrm{Sp}(2n)\cap O(2n)$, if $1$ and $-1$ are eigenvalues of $Q$, their multiplicities must be even, for example see \cite{Wiggins}. Since $QJ=JQ$, there exists an orthogonal matrix $P$ such that
\begin{equation}\label{a2.3}
P^{\top}JP=\diag\left(\left(
\begin{array}{ccc}
0&1\\
-1&0
\end{array}
\right),\cdots,\left(
\begin{array}{ccc}
0&1\\
-1&0
\end{array}
\right)\right),
\end{equation}
and
\begin{align}\label{a2.4}
P^\top QP=\diag(M_1,\cdots,M_{n}),
\end{align}
where
\begin{equation*}
\begin{split}
M_j&=\left(
\begin{array}{ccc}
\cos \theta_j&\sin \theta_j\\
-\sin\theta_j&\cos \theta_j
\end{array}
\right),
\end{split}\ \ 0\leq\theta_j< 2\pi,\ 1\leq j\leq n.
\end{equation*}

For a constant $p>1$, denote
$$X=\{x\in L^p([0,T],\mathbf{R}^{2n}): x(t+T)=Qx(t)\ {\rm for\ all}\ t\in\mathbf{R}\}$$
with the norm $\|\cdot\|$ defined by
$\|x\|=\left(\int_0^T|x(t)|^p\mathrm{d}t\right)^{\frac{1}{p}}.$
Then it is easy to see that $X$ is a Banach space.
For $s\in\mathbf{R}$, consider the group action on $X$ by
$$\mathcal{Q}(s)x(t)=x(t+s).$$
Clearyly, $X$ is $\mathcal{Q}(s)$-invariant, that is, if $x\in X$, then $\mathcal{Q}(s)x\in X$ for all $s\in\mathbf{R}$.

Now we give the $\mathcal{Q}(s)$-index:
\begin{definition}\label{d1}
A $\mathcal{Q}(s)$ index $\ind (\Gamma)$ of an invariant closed subset $\Gamma\subset X$ is the smallest integer $k$ such that there exists a
$$\Phi=(\Phi_1^\top,\cdots,\Phi_m^\top)^\top\in C(\Gamma,\mathbf{C}^{k}\setminus\{0\})$$
with $\Phi_j\in C(\Gamma,\mathbf{C}^{k_j})$ and
\begin{eqnarray}\label{02.4}
\Phi_j(\mathcal{Q}(s)x)=e^{\frac{\sqrt{-1}(2\pi n_j+\theta_{p_j})}{T}s}\Phi_j(x)\ \ \forall s\in\mathbf{R},\  x\in\Gamma,
\end{eqnarray}
where $k_j\in\mathbf{N}_+$, $k_1+\cdots+k_m=k$, $n_j\in\mathbf{Z}$, $1\leq p_j\leq n$ and $2\pi n_j+\theta_{p_j}\neq 0$ for $1\leq j\leq m$.
If such a $\Phi$ does not exist, define $\ind (\Gamma)=+\infty$, and if $\Gamma=\emptyset$, define $\ind (\Gamma)=0$.
\end{definition}

We need to show:
\begin{lemma}\label{p1}
The $\mathcal{Q}(s)$-index is an index in the sense of Definition \ref{def1}.
\end{lemma}
\begin{proof}
{\rm (i)} Assume $x\in X\backslash\fix\{\mathcal{Q}(s)\}$.
Consider the function
\begin{eqnarray}
y(t)=P_0^{-1}x(t)=(y_{1,1}(t),y_{1,2}(t),\cdots,y_{n,1}(t),y_{n,2}(t))^\top,
\end{eqnarray}
where $y_{j,k}\in\mathbf{C}$ for $1\leq j\leq n$, $k=0,1$, and $P_0$ is the unitary matrix such that
$$P_0^{-1}QP_0=\diag(e^{\sqrt{-1}\theta_1},e^{-\sqrt{-1}\theta_1},\cdots,e^{\sqrt{-1}\theta_n},e^{-\sqrt{-1}\theta_n}).$$
Then one has
\begin{align}
y(t+T)&=P_0^{-1}x(t+T)=P_0^{-1}QP_0P_0^{-1}x(t)\nonumber\\
&=(e^{\sqrt{-1}\theta_1}y_{1,1}(t),e^{-\sqrt{-1}\theta_1}y_{1,2}(t),\cdots,e^{\sqrt{-1}\theta_n}y_{n,1}(t),e^{-\sqrt{-1}\theta_n}y_{n,2}(t))^\top.
\end{align}
Hence for $1\leq j\leq n$,
$$e^{-\frac{\sqrt{-1}\theta_{j}}{T}(t+T)}y_{j,1}(t+T)=e^{-\frac{\sqrt{-1}\theta_{j}}{T}t}y_{j,1}(t),\ \ t\in\mathbf{R}.$$
Since $x\notin \fix\{\mathcal{Q}(s)\}$,
there exist a component $y_{j_0,1}(t)$ and $n_{j_0}\in\mathbf{Z}$ with $2\pi n_{j_0}-\theta_{{j_0}}\neq0$, such that
\begin{eqnarray}
\Phi(x)=\int_0^Te^{\frac{\sqrt{-1}(2\pi n_{j_0}-\theta_{j_0})}{T}s}y_{j_0,1}(s)\mathrm{d}s\neq0.
\end{eqnarray}
Clearly, $\Phi$ is continuous and
$$\Phi(\mathcal{Q}(s)x)=e^{\frac{\sqrt{-1}(-2\pi n_{j_0}+\theta_{j_0})}{T}s}\Phi(x).$$
For $z\in\overline{\cup_{s\in\mathbf{R}}\mathcal{Q}(s)x}$, there exist $s_k\in\mathbf{R}$ for $k\geq 1$ such that
$\lim\limits_{k\rightarrow\infty}\mathcal{Q}(s_k)x=z$. Then
$$|\Phi(z)|=\lim\limits_{k\rightarrow\infty}|\Phi(\mathcal{Q}(s_k)x)|=|\Phi(x)|\neq0,$$
and for each $s\in\mathbf{\mathbf{R}}$,
\begin{align*}\Phi(\mathcal{Q}(s)z)&=\lim\limits_{k\rightarrow\infty}\Phi(\mathcal{Q}(s)\mathcal{Q}(s_k)x)\\
&=e^{\frac{\sqrt{-1}(-2\pi n_{j_0}+\theta_{j_0})}{T}s}\lim\limits_{k\rightarrow\infty}\Phi(\mathcal{Q}(s_k)x)\\
&=e^{\frac{\sqrt{-1}(-2\pi n_{j_0}+\theta_{j_0})}{T}s}\Phi(z).
\end{align*}
Hence $\ind(\overline{\cup_{s\in\mathbf{R}}\mathcal{Q}(s)x})=1$.

{\rm (ii)} If $\ind(\Gamma_2)=\infty$, the result is obvious. If $\ind(\Gamma_2)=k<\infty$, there exists a
$$\Phi=(\Phi_1^\top,\cdots,\Phi_m^\top)^\top\in C(\Gamma_2,\mathbf{C}^{k}\setminus\{0\})$$
such that $\Phi_j\in C(\Gamma_2,\mathbf{C}^{k_j})$ and
\begin{eqnarray}\label{3.3}
\Phi_j(\mathcal{Q}(s)x)=e^{\frac{\sqrt{-1}(2\pi n_j+\theta_{p_j})}{T}s}\Phi_j(x)\ \ \forall s\in\mathbf{R},\  x\in\Gamma_2,
\end{eqnarray}
where $k_j\in\mathbf{N}_+$, $k_1+\cdots+k_m=k$, $n_j\in\mathbf{Z}$, $1\leq p_j\leq n$ and $2\pi n_j+\theta_{p_j}\neq 0$ for $1\leq j\leq m$.
Define $\tilde{\Phi}(x)=\Phi\circ R(x)$ for $x\in\Gamma_1$. Then
\begin{align*}
\tilde{\Phi}_j(\mathcal{Q}(s)x)&=\Phi_j\circ R(\mathcal{Q}(s)x)=\Phi_j(\mathcal{Q}(s)R(x))\\
&=e^{\frac{\sqrt{-1}(2\pi n_j+\theta_{p_j})}{T}s}\Phi_j\circ R(x)\\
&=e^{\frac{\sqrt{-1}(2\pi n_j+\theta_{p_j})}{T}s}\tilde{\Phi}(x),
\end{align*}
which yields $\ind(\Gamma_1)\leq\ind(\Gamma_2)$.

{\rm (iii)} Since the inclusion map is equivariant, the  monotonicity property is obvious.

{\rm (iv)} When $\Gamma_1=\emptyset$, the result is obvious. When $\Gamma_1\neq\emptyset$, for each $x\in\Gamma_1$, similar to the proof of (i), the function
\begin{eqnarray}\label{02.5}
y(t)=P_0^{-1}x(t)=(y_{1,1}(t),y_{1,2}(t),\cdots,y_{n,1}(t),y_{n,2}(t))^\top,
\end{eqnarray}
has component $y_{j_0,1}(t)$ and $n_{j_0}\in\mathbf{Z}$ with $2\pi n_{j_0}-\theta_{{j_0}}\neq0$, such that
\begin{eqnarray}
\int_0^Te^{\frac{\sqrt{-1}(2\pi n_{j_0}-\theta_{j_0})}{T}s}y_{j_0,1}(s)\mathrm{d}s\neq0.
\end{eqnarray}
By the continuity, there exists a $\mathcal{Q}(s)$ invariant open neighbourhood $U_{x}$ of $x$ such that for each $\tilde{x}\in \overline{U}_{x}$ and $\tilde{y}=P_0^{-1}\tilde{x}$,
\begin{eqnarray}
\int_0^Te^{\frac{\sqrt{-1}(2\pi n_{0}-\theta_{j_0})}{T}s}\tilde{y}_{j_0,1}(s)\mathrm{d}s\neq0.
\end{eqnarray}
Since $\Gamma_1$ is compact, there exist finite $U_{x^l}$ for $1\leq l\leq r$ such that
$$\Gamma_1\subset\cup_{l=1}^r U_{x^l}.$$
For each $x\in \overline{U}_{x^l}$, there exists a component $y_{j_l,1}$ of $y^l=P_0^{-1}x^l$ and $n_{j_l}\in\mathbf{Z}$ with $2\pi n_{j_l}-\theta_{{j_l}}\neq0$, such that
$$\Upsilon_l(x)=\int_0^Te^{\frac{\sqrt{-1}(2\pi n_{l}-\theta_{j_l})}{T}s}y_{j_l,1}(s)\mathrm{d}s\neq0.$$
Let
$$\Upsilon(x)=(\Upsilon_1(x),\cdots,\Upsilon_r(x))^\top.$$
Then for each $x\in\Gamma_1$, $\Upsilon(x)\neq0$ and
$$\Upsilon_l(\mathcal{Q}(s)x)=e^{\frac{\sqrt{-1}(-2\pi n_l+\theta_{j_l})}{T}s}\Upsilon_l(x),$$
for $1\leq l\leq r$. Hence $\ind(\Gamma_1)<\infty$.

Assume $\ind(\Gamma_1)=k$.
Then there exists a mapping
$$\Phi=(\Phi_1^\top,\cdots,\Phi_m^\top)^\top\in C(\Gamma_1,\mathbf{C}^{k}\setminus\{0\})$$
satisfying Definition \ref{d1}. Since $\Gamma_1$ is a closed set, by Tietze's theorem there exists a continuous extension $\tilde{\Phi}_j$ of $\Phi_j$ over $X$ for each $1\leq j\leq m$. Now define a mapping $\Psi=(\Psi_1^\top,\cdots,\Psi_m^\top)^\top$ by
\begin{eqnarray}\label{3.24}
\Psi_j(x)=\lim\limits_{t\rightarrow\infty}\frac{1}{t}\int_{0}^{t}e^{\frac{\sqrt{-1}(-2\pi n_j-\theta_{p_j})}{T}s}\tilde{\Phi}_j(\mathcal{Q}(s)x)\mathrm{d}s.
\end{eqnarray}
Since $\tilde{\Phi}_j(\mathcal{Q}(s)x)$ is an almost periodic function on $s$, \eqref{3.24} is well defined.
Clearly, $\Psi$ is continuous, $\Psi(x)=\Phi(x)$ for $x\in\Gamma_1$ and
$$\Psi_j(\mathcal{Q}(s)x)=e^{\frac{\sqrt{-1}(2\pi n_j+\theta_{p_j})}{T}s}\Psi(x)\ \ \forall x\in X,\ s\in\mathbf{R}.$$
Let
$$\Gamma_\delta=\{x\in X:\ \dist(x,\Gamma_1)\leq\delta\}.$$
Then it is easy to see that $\Gamma_\delta$ is $\mathcal{Q}(s)$ invariant. Since $\Psi$ is continuous and $\Psi(x)\neq0$ for $x\in\Gamma_1$, there exists a $\delta>0$ such that $\Psi(x)\neq0$ for $x\in\Gamma_\delta$ which yields $\ind (\Gamma_\delta)\leq k$. By the monotonicity property of the index, one has $\ind (\Gamma_\delta)\geq k$. Thus $\ind (\Gamma_1)=\ind (\Gamma_\delta).$

{\rm (v)}
If $\ind(\Gamma_1)=\infty$ or $\ind(\Gamma_2)=\infty$, the result is obvious. Assume $\ind (\Gamma_1)=k_1<\infty$ and $\ind(\Gamma_2)=k_2<\infty$, then there exist
$$\Phi^j=((\Phi^j_1)^\top,\cdots,(\Phi^j_{m_j})^\top)^\top\in C(\Gamma_j,\mathbf{C}^{k_j}\setminus\{0\})$$
for $j=1,\ 2$ satisfying \eqref{02.4}. Similar to the proof of (iv), there exist continuous extensions $\Psi_j(x)$ of $\Phi^j(x)$ for $j=1, 2$ over $X$ satisfying \eqref{02.4}.  Define
$$\Psi:X\rightarrow\mathbf{C}^{k_1+k_2}$$
by
$$\Psi(x)=(\Psi_1^\top(x),\Psi_2^\top(x))^\top.$$
Then $\Psi(x)\neq0$ for every $x\in \Gamma_1\cup\Gamma_2$ and satisfies \eqref{02.4}, which yields
$\ind(\Gamma_1\cup\Gamma_2)\leq k_1+k_2$.
\end{proof}
\begin{definition}
A $\mathcal{Q}(s)$ orbit of $X$ is the set $\{\mathcal{Q}(s)x: s\in\mathbf{R}\}$ for $x\in X$.
\end{definition}
For compact group actions $\mathbf{Z}_2$ and $S^1$, an important property is that if the index is greater than or equal to $2$, the set contains infinitely many orbits. We can't obtain such a strong result for $\mathcal{Q}(s)$ index. Fortunately, we have the following lemma.
\begin{lemma}\label{p2}
Assume $\Gamma$ is an invariant subset of $X$, such that
$$\Gamma\cap\fix\{\mathcal{Q}(s)\}=\emptyset.$$
If $\ind \Gamma=k>0$, then there
exist at least $k$ $\mathcal{Q}(s)$-orbits on $\Gamma$.
\end{lemma}
\begin{proof}
Assume $\Gamma$ only contains $k-1$ orbits: $\mathcal{Q}(s)x^1, \cdots, \mathcal{Q}(s)x^{k-1}$.
Then
$$\Gamma=\bigcup_{i=1}^{k-1}\overline{\{\mathcal{Q}(s)x^i: s\in\mathbf{R}\}}.$$
As in the proof of Lemma \ref{p1}, for each $x^l\in\Gamma$ and $y^l=P_0^{-1}x^l$, there exist a component $y^l_{j_l,1}$ and $n_l\in\mathbf{Z}$ with
$2\pi n_l+\theta_{j_l}\neq0$ such that
$$\Phi_l(x):=\int_0^Te^{\frac{\sqrt{-1}(-2\pi n_{l}-\theta_{j_l})}{T}s}y^l_{j_l,1}(s)\mathrm{d}s\neq0,$$
for $1\leq l\leq k-1$. Then $\Phi_l(x)\neq0$ for $x\in\overline{\{\mathcal{Q}(s)x^l: s\in\mathbf{R}\}}$ and
$$\Phi_l(\mathcal{Q}(s)x)=e^{\frac{\sqrt{-1}(2\pi n_{l}+\theta_{j_l})}{T}s}\Phi_l(x)\ \  \forall s\in\mathbf{R},\ x\in\Gamma.$$
Hence
$\Phi=(\Phi_1,\cdots,\Phi_{k-1})^\top\in C(\Gamma,\mathbf{C}^{k-1}\setminus\{0\})$
and $\ind(\Gamma)\leq k-1$. This contradicts the assumption and the lemma is proved.
\end{proof}

Let $E$ be a real continuous function on $X$, for $j\geq 1$, define
\begin{eqnarray*}
\mathcal{A}_j=\{A\subset X: \ A\ {\rm is}\ \mathcal{Q}(s){\rm -invariant\ and}\ \ind A\geq j\},
\end{eqnarray*}
\begin{eqnarray}\label{02.15}
c_j=\inf\limits_{A\in\mathcal{A}_j}\sup\limits_{A}E.
\end{eqnarray}
It is easy to see that $\mathcal{A}_j\subset\mathcal{A}_{j-1}$ for $j\geq2$, and then
\begin{eqnarray}\label{2.14}
-\infty\leq c_1\leq c_2\leq\cdots\leq+\infty.
\end{eqnarray}
Denote
$$K_c=\{z\in X: E'(z)=0\ \ {\rm and}\ E(z)=c\},$$
$$E^c=\{z\in X: E(z)\leq c\},$$
where $c$ is a constant.

We are in a position to state an abstract critical point theorem on the $\mathcal{Q}(s)$-index.
\begin{theorem}\label{th3.1}
Assume $E\in C^1(X,\mathbf{R})$ is $\mathcal{Q}(s)$-invariant and satisfies (P.-S.) (Palais-Smale condition). If $-\infty<c_j<+\infty$ and $K_{c_j}\cap\fix\{\mathcal{Q}(s)\}=\emptyset$, then $c_j$
is a critical value of $E$. Moreover, if $c_i=c_j$ for some $i\leq j$, then
$$\ind K_{c_i}\geq j-i+1.$$
\end{theorem}
Before proving Theorem \ref{th3.1}, we need the concept of ``pseudo-gradient".
\begin{definition}\label{def3.3}
Assume $Y$ is a Banach space, $\varphi\in C^1(Y,\mathbf{R})$ and
$$\tilde{Y}=\{u\in Y:\varphi'(u)\neq0\}.$$
A pseudo-gradient vector field for $\varphi$ on $\tilde{Y}$ is a locally Lipschitz continuous mapping $\nu:\tilde{Y}\rightarrow Y$ such that, for every
$u\in\tilde{Y}$, one has
\begin{eqnarray*}
\|\nu(u)\|\leq2\|\varphi'(u)\|,
\end{eqnarray*}
\begin{eqnarray*}
\langle\nu(u),\varphi'(u)\rangle\geq\|\varphi'(u)\|^2.
\end{eqnarray*}
\end{definition}

We need the following lemmas.
\begin{lemma}\label{lem3.1}(see \cite{Mawhin}).
Under the assumption of Definition \ref{def3.3}, there exists a pseudo-gradient vector field for $\varphi$ on $\tilde{Y}$.
\end{lemma}
\begin{lemma}
Assume $E\in C^1(X,R)$ is $\mathcal{Q}(s)$-invariant, then there exists an equivariant pseudo-gradient vector filed for
$E$ on $\tilde{X}=\{z\in X: E'(z)\neq0\}$. That is, $v(\mathcal{Q}(s)z)=\mathcal{Q}(s)v(z)$ for every $z\in\tilde{X}$ and $s\in\mathbf{R}$.
\end{lemma}
\begin{proof}
By Lemme \ref{lem3.1}, there exists a pseudo-gradient vector field $w:\tilde{X}\rightarrow X$.
Now define $v:\tilde{X}\rightarrow X$ by
$$\nu(z)=\lim\limits_{r\rightarrow\infty}\frac{1}{r}\int_0^r\mathcal{Q}(-s)w(\mathcal{Q}(s)z)\mathrm{d}s.$$
It is easy to see that $\mathcal{Q}(-s)w(\mathcal{Q}(s)z)$ is almost periodic on $s$, and $\nu$ is well defined.
Then, for $\tau\in\mathbf{R}$, we have
\begin{align*}
\nu(\mathcal{Q}(\tau)z)&=\lim\limits_{r\rightarrow\infty}\frac{1}{r}\int_{0}^r\mathcal{Q}(-s)w(\mathcal{Q}(s+\tau)z)\mathrm{d}s\\
&=\mathcal{Q}(\tau)\lim\limits_{r\rightarrow\infty}\frac{1}{r}\int_{0}^r\mathcal{Q}(-s-\tau)w(\mathcal{Q}(s+\tau)z)\mathrm{d}s\\
&=\mathcal{Q}(\tau)\lim\limits_{r\rightarrow\infty}\frac{1}{r}\int_{0}^r\mathcal{Q}(-t)w(\mathcal{Q}(t)z)\mathrm{d}t\\
&=\mathcal{Q}(\tau)\nu(z).
\end{align*}
By a simple calculation, we obtain
\begin{align*}
\|\nu(z)\|\leq\sup\limits_{s\in\mathbf{R}}\|w(\mathcal{Q}(s)z)\|\leq2\sup\limits_{s\in\mathbf{R}}\|E'(\mathcal{Q}(s)z)\|=2\|E'(z)\|,
\end{align*}
\begin{align*}
\langle\nu(z),E'(z)\rangle&=\lim\limits_{r\rightarrow\infty}\frac{1}{r}\int_{0}^r\left\langle \mathcal{Q}(-s)w(\mathcal{Q}(s)z),E'(z)\right\rangle\mathrm{d}s\\
&=\lim\limits_{r\rightarrow\infty}\frac{1}{r}\int_{0}^r\left\langle w(\mathcal{Q}(s)z),E'( \mathcal{Q}(s)z)\right\rangle\mathrm{d}s\\
&\geq\lim\limits_{r\rightarrow\infty}\frac{1}{r}\int_{0}^r\|E'(\mathcal{Q}(s)z)\|^2\mathrm{d}s\\
&=\|E'(z)\|^2.
\end{align*}
Now it only remains to show $\nu$ is locally Lipschitz continuous.
For $z\in X$, denote $\Gamma=\{\mathcal{Q}(s)z:s\in\mathbf{R}\}$, then the closure $\overline{\Gamma}$ is the hull of $z$ and so is compact. It is easy to see that there exists a $\delta>0$ such that $w$ is Lipschitz continuous on
$$\Gamma_\delta=\{z\in X:\dist(z,\overline{\Gamma})<\delta\}.$$
Clearly, $\Gamma_\delta$ is $\mathcal{Q}(s)$-invariant, and for each $z_1,z_2\in\Gamma_\delta$ we have
\begin{align*}
\|\nu(z_1)-\nu(z_2)\|&\leq\lim\limits_{r\rightarrow\infty}\frac{1}{r}\int_{0}^r\|\mathcal{Q}(-s)(w(\mathcal{Q}(s)z_1)-w(\mathcal{Q}(s)z_2))\|\mathrm{d}s\\
&=\lim\limits_{r\rightarrow\infty}\frac{1}{r}\int_{0}^r\|(w(\mathcal{Q}(s)z_1)-w(\mathcal{Q}(s)z_2))\|\mathrm{d}s\\
&\leq L\lim\limits_{r\rightarrow\infty}\frac{1}{r}\int_{0}^r\|\mathcal{Q}(s)(z_1-z_2)\|\mathrm{d}s\\
&=L\|z_1-z_2\|,
\end{align*}
where $L$ is the Lipschitz constant of $w$ on $\Gamma_\delta$.
\end{proof}

\begin{lemma}\label{lem3}
Assume $E\in C^1(X,\mathbf{R})$ is $\mathcal{Q}(s)$-invariant and satisfies (P.-S.), $U$ is an open invariant neighbourhood of $K_c$. Then for each $\epsilon>0$, there exist $\varepsilon\in(0,\epsilon]$ and $\eta\in C([0,1]\times X,X)$,
such that
$$\eta(1,E^{c+\varepsilon}\setminus U)\subset E^{c-\varepsilon},$$
and for each $t\in[0,1]$,
$$\eta(t,z)=z,\ \ {\rm if}\ z\notin E^{-1}([c-\epsilon,c+\epsilon]).$$
Moreover,
$$\eta(t,\mathcal{Q}(s)z)=\mathcal{Q}(s)\eta(t,z),\ {\rm for\ each\ }z\in X,\ t\in[0,1],\ s\in\mathbf{R}.$$
\end{lemma}
\begin{proof}
Firstly we claim that for each given $\epsilon>0$, there exists a $\varepsilon\in(0,\frac{\epsilon}{2}]$, such
that if $z\in E^{-1}([c-2\varepsilon,c+2\varepsilon])\cap (U^C)_{2\sqrt{\varepsilon}}$, then
\begin{align}\label{3.4}
\|E'(z)\|\geq4\sqrt{\varepsilon},
\end{align}
where $U^C$ denotes the complement of $U$, and $(U^C)_{2\sqrt{\varepsilon}}$ is the $2\sqrt{\varepsilon}$ neighbourhood of
$U^C$. If such $\varepsilon$ does not exist, there is a sequence $\{z_k\}$ such that
$$z_k\in(U^C)_{\frac{2}{\sqrt{k}}},\ c-\frac{2}{k}\leq E(z_k)\leq c+\frac{2}{k},$$
and
$$\|E'(z_k)\|<\frac{4}{\sqrt{k}}.$$
Since $E$ satisfies (P.-S.), it has a convergent subsequence, without loss of generality, still denoted by $\{z_k\}$.
Assume $\lim\limits_{k\rightarrow\infty}z_k=z$, then
$$E(z)=c,\ \ E'(z)=0,$$
and $z\in K_c\cap U^C=\emptyset$, a contradiction.

Now, let
$$A=\{z\in X:z\in E^{-1}([c-2\varepsilon,c+2\varepsilon])\cap (U^C)_{2\sqrt{\varepsilon}}\},$$
$$B=\{z\in X:z\in E^{-1}([c-\varepsilon,c+\varepsilon])\cap (U^C)_{\sqrt{\varepsilon}}\}\subset A,$$
$$\psi(z)=\frac{\dist(z,A^C)}{\dist(z,A^C)+\dist(z,B)}.$$
Then
$0\leq \psi(z)\leq1,$
and $\psi(z)=1$ if $z\in B$, $\psi(z)=0$ if $z\in A^C$.
Define a continuous function $g$ on $\mathbf{R}^{2n}$,
\begin{equation}
g(z)=\bigg\{
\begin{array}{ccc}
-\psi(z)\frac{\nu(z)}{\|\nu(z)\|},\ {\rm if}\ z\in A,\\
0,\ \ \ \ \ \ \ \ \ \ \ \ \ \ \  {\rm if}\ z\notin A,
\end{array}
\end{equation}
where $\nu(z)$ is an equivariant pseudo-gradient vector filed for $E$. Since $g$ is locally Lipschitz continuous and bounded, the following Cauchy problem has a unique solution $x(\cdot,z)$ defined on $[0,\infty)$:
\begin{equation}\label{3.6}
\bigg\{
\begin{array}{ccc}
x'=g(x),\\
x(0)=z.
\end{array}
\end{equation}
Let
\begin{align*}
\eta(t,z)=x(\sqrt\varepsilon t,z),\ \ t\in[0,1].
\end{align*}
Then it is easy to see that $\eta(\cdot,\cdot)$ is continuous. Since
\begin{align*}
\|x(t,z)-z\|=\|\int_0^tg(x(s,z))\mathrm{d}s\|\leq\int_0^t\|g(x(s,z))\|\mathrm{d}s\leq t,
\end{align*}
we have
$$x(t,U^C)\subset(U^C)_{\sqrt\varepsilon}\ \ \ \forall t\in[0,\sqrt\varepsilon].$$
By the definition of $g$, for $t\in[0,\sqrt\varepsilon]$ and $x(t,z)\in A$, we obtain
\begin{align*}
\frac{\mathrm{d}}{\mathrm{d}t}E(x(t,z))&=\langle E'(x(t,z)),g(x(t,z))\rangle\\
&=\left\langle E'(x(t,z)),-\psi(x(t,z))\frac{\nu(x(t,z))}{\|\nu(x(t,z))\|}\right\rangle\\
&=\frac{-\psi(x(t,z))}{\|\nu(x(t,z))\|}\left\langle E'(x(t,z)),\nu(x(t,z))\right\rangle\\
&\leq\frac{-\psi(x(t,z))}{\|\nu(x(t,z))\|}\|E'(x(t,z)))\|^2\\
&\leq0.
\end{align*}
For $z\in E^{c+\varepsilon}\setminus U$, if $E(x(\tau,z))<c-\varepsilon$ for some $\tau\in[0,\sqrt\varepsilon]$,
then $E(x(\sqrt\varepsilon,z))<c-\varepsilon$, and $\eta(1,z)\in E^{c-\varepsilon}$. If such $\tau\in[0,\sqrt\varepsilon]$ doesn't exist, then
$$x(t,z)\in B\ \ \ \forall t\in[0,\sqrt\varepsilon].$$
From \eqref{3.4}, we have
\begin{align*}
E(x(\sqrt{\varepsilon},z))&=E(z)+\int_0^{\sqrt\varepsilon}\frac{\mathrm{d}}{\mathrm{d}t}E(x(t,z))\mathrm{d}t\\
&=E(z)+\int_0^{\sqrt\varepsilon}\langle E'(x(t,z)),g(x(t,z))\rangle\mathrm{d}t\\
&=E(z)+\int_0^{\sqrt\varepsilon}\left\langle E'(x(t,z)),-\frac{\nu(x(t,z))}{\|\nu(x(t,z))\|}\right\rangle\mathrm{d}t\\
&\leq c+\varepsilon-\int_0^{\sqrt\varepsilon}\frac{\|E'(x(t,z))\|^2}{\|\nu(x(t,z))\|}\mathrm{d}t\\
&\leq c-\varepsilon.
\end{align*}
Then $\eta(1,z)\in E^{c-\varepsilon}$, and $\eta(1,E^{c+\varepsilon}\setminus U)\subset E^{c-\varepsilon}$.
It is easy to see that $\psi(\mathcal{Q}(s)z)=\psi(z)$, which yields that $g(\mathcal{Q}(s)z)=\mathcal{Q}(s)g(z)$.
Then,
\begin{align*}
\eta(t,\mathcal{Q}(s)z)=x(\sqrt\varepsilon t,\mathcal{Q}(s)z)&=\mathcal{Q}(s)z+\int_0^{\sqrt\varepsilon t}g(x(\tau,\mathcal{Q}(s)z))\mathrm{d}\tau\\
&=\mathcal{Q}(s)z+\mathcal{Q}(s)\int_0^{\sqrt\varepsilon t}g(\mathcal{Q}(-s)x(\tau,\mathcal{Q}(s)z))\mathrm{d}\tau.\\
\end{align*}
Thus, we have
\begin{align*}
\mathcal{Q}(-s)\eta(t,\mathcal{Q}(s)z)=\mathcal{Q}(-s)x(\sqrt\varepsilon t,\mathcal{Q}(s)z)=z+\int_0^{\sqrt\varepsilon t}g(\mathcal{Q}(-s)x(\tau,\mathcal{Q}(s)z))\mathrm{d}\tau.
\end{align*}
Since the solution of \eqref{3.6} is unique, we get
$$\mathcal{Q}(-s)x(t,\mathcal{Q}(s)z)=x(t,z),$$
and thus
$$\eta(t,\mathcal{Q}(s)z)=\mathcal{Q}(s)\eta(t,z),\ \ {\rm for\ each}\ t\in[0,1],\ s\in\mathbf{R}.$$
\end{proof}

Now we give the proof of Theorem \ref{th3.1}.
\begin{proof}[Proof of Theorem \ref{th3.1}]

Assume $c_i=c_j=c$ for some $1\leq i\leq j\leq n$. By Lemma \ref{p2}, it suffices to prove that
$$\ind K_c\geq j-i+1.$$
Since $E$ satisfies (P.-S.) and $\mathcal{Q}(s)$-invariant, we see $K_c$ is compact and $\mathcal{Q}(s)$-invariant. By the continuity property of index, there exists a closed $\mathcal{Q}(s)$-invariant neighbourhood $N$ of $K_c$
such that
$$\ind N=\ind K_c,$$
and the interior $U$ of $N$ is an open $\mathcal{Q}(s)$-invariant neighbourhood of $K_c$. Take $A\in\mathcal{A}_j$
such that
$$\sup\limits_{A}E\leq c+\varepsilon,$$
and denote $\Omega=A\setminus U$.
By the monotonicity and subadditivity properties of index, we obtain
\begin{align*}
j\leq\ind A\leq\ind(\Omega\cup N)\leq\ind\Omega+\ind N=\ind \Omega+\ind K_c.
\end{align*}
Now we apply Lemma \ref{lem3} with $\Lambda=\eta(1,\Omega)\subset E^{c-\varepsilon}$. Then $\Lambda$ is $\mathcal{Q}(s)$-invariant and
$$\sup\limits_{\Lambda}E\leq c-\varepsilon.$$
By the definition of $c_i=c$, one has
$$\ind \Lambda\leq i-1.$$
Since $\eta(1,\cdot)$ is equivariant, by the mapping property of index, we have
$$\ind \Omega\leq\ind\Lambda\leq i-1.$$
Then
$$\ind K_c\geq j-\ind\Omega\geq j-i+1.$$
\end{proof}

Let
\begin{equation}
\tilde{\theta}_j=\bigg\{
\begin{array}{ccc}
&\theta_j,\ \ {\rm if}\ \theta_j\neq0,\\
&2\pi,\ \  {\rm if}\ \theta_j=0,
\end{array}\ \ \ \ \ \ \ 1\leq j\leq n.
\end{equation}
Clearly, $0<\tilde{\theta}_j\leq2\pi$, without loss of generality, let's assume that
\begin{eqnarray}
\tilde{\theta}_1\leq\tilde{\theta}_2\leq\cdots\leq\tilde{\theta}_n.
\end{eqnarray}
For $1\leq j\leq n$, let
\begin{equation*}
\begin{split}
M_{j}(t)&=\left(
\begin{array}{ccc}
\cos \frac{\tilde{\theta}_j}{T}t&\sin \frac{\tilde{\theta}_j}{T}t\\
-\sin \frac{\tilde{\theta}_j}{T}t&\cos \frac{\tilde{\theta}_j}{T}t
\end{array}
\right),
\end{split}\ \ t\in\mathbf{R}.
\end{equation*}

The following result is useful to us.
\begin{theorem}\label{th2}
Let $M=\{Q(t)\xi:\xi\in\mathbf{R}^{2n},\ |\xi|=1\}\subset X$, where
$$Q(t)=P\diag\{M_{1}(t),\cdots,M_{n}(t)\}P^\top$$
and $P$ is given by \eqref{a2.3}. Then
$M$ is $\mathcal{Q}(s)$-invariant and $\ind (M)=n$.
\end{theorem}
To prove the theorem, we need the following lemma about $S^1$ action on $\mathbf{R}^{2k}$.
\begin{lemma}[See Theorem 5.5 of \cite{Mawhin}]\label{lemM}
Let $\{T(\theta)\}_{\theta\in S^1}$ be an action of $S^1$ over $\mathbf{R}^{2k}$ such that $\fix(S^1)=\{0\}$ and let $D$ be an open
bounded invariant neighbourhood of $0$. If $\Phi\in C(\partial D,\mathbf{C}^{k-1})$  and $m\in \mathbf{Z}\setminus\{0\}$ with
$$\Phi(T(\theta)z)=e^{\sqrt{-1}m\theta}\Phi(z),\ \ \theta\in S^1,\ \ z\in\partial D,$$
then $0\in\Phi(\partial D)$.
\end{lemma}

Now we give the proof of
Theorem \ref{th2}.
\begin{proof}[Proof of Theorem \ref{th2}]
Take unitary matrix $\bar{P}$
such that
$$\bar{P}^{-1}Q(t)\bar{P}=\diag\left(e^{\frac{\sqrt{-1}\tilde{\theta}_{1}}{T}t},e^{-\frac{\sqrt{-1}\tilde{\theta}_{1}}{T}t},\cdots,e^{\frac{\sqrt{-1}\tilde{\theta}_{n}}{T}t},e^{-\frac{\sqrt{-1}\tilde{\theta}_{n}}{T}t}\right).$$
For each $z=Q(t)\xi\in M$ with $|\xi|=1$, denote
$$y=(y_{1,1},y_{1,2},\cdots, y_{n,1},y_{n,2})^\top=\bar{P}^{-1}\xi,$$
where $y_{j,i}\in\mathbf{C}$ for $1\leq j\leq n$ and $i=1,2$.
Let $\Psi_j(z)=y_{j,1}$ for $1\leq j\leq n$ and
$$\Psi(z)=(\Psi_1(z),\Psi_2(z)\cdots,\Psi_n(z))^\top.$$
Then $\Psi\in C(M, \mathbf{C}^{n}\setminus\{0\})$ and
$$\Psi_j(\mathcal{Q}(s)z)=e^{\frac{\sqrt{-1}\tilde{\theta}_{j}}{T}s}\Psi_j(z),\ \ \ 1\leq j\leq n.$$
It follows that $\ind(M)\leq n$.
Assume $\ind(M)=q<n$. Then there exists a mapping
\begin{align}\label{002.17}
\Phi=(\Phi_{1}^\top,\cdots,\Phi_{m}^\top)^\top\in C( M,\mathbf{C}^q\setminus\{0\})
\end{align}
such that $\Phi_j=(\Phi_{j,1},\cdots,\Phi_{j,q_j})^\top\in C(M,\mathbf{C}^{q_j})$ with $q_1+\cdots q_m=q$ and
$$\Phi_j(\mathcal{Q}(s)x)=e^{\frac{\sqrt{-1}(2k_j\pi+\theta_{r_j})}{T}s}\Phi_j(x)$$
for $1\leq j\leq m,\ k_j\in\mathbf{Z},\ 1\leq r_j\leq n$, $2k_j\pi+\theta_{r_j}\neq0$. Rewrite $(z_1,\cdots,z_n)=(\tilde{z}_1,\cdots,\tilde{z}_l)$
such that
$$\mathcal{Q}(s+T_j)\tilde{z}_j=\mathcal{Q}(s)\tilde{z}_j\ \ \forall s\in\mathbf{R},\ \ 1\leq j\leq l,$$
where $T_j>0$ is a constant and
\begin{align}\label{3.16}
\tilde{k}_1T_j+\tilde{k}_2T_r\neq0\ \  \forall j\neq r, \ \ \ \tilde{k}_1,\ \tilde{k}_2\in\mathbf{Z}\setminus\{0\}.
\end{align}
Then for any $x^j=(\tilde{x}^j_1,\cdots,\tilde{x}^j_l)\in M$ with $\tilde{x}^j_a=0$, $1\leq a\leq l$, $a\neq j$, one has
$$\mathcal{Q}(s+T_j)x^j=\mathcal{Q}(s)x^j,\ \ \ 1\leq j\leq l.$$
Since
$$\Phi_a(\mathcal{Q}(T_j)x^j)=\Phi_a(x^j)=e^{\frac{\sqrt{-1}(2k_a\pi+\theta_{r_a})}{T}T_j}\Phi_a(x^j)\ \ {\rm for}\ \ 1\leq a\leq m,$$
one has
$$\Phi_a(x^j)=0,\ \ \  {\rm if}\ \frac{(2k_a\pi+\theta_{r_a})T_j}{T}\neq2\pi k\ \ \ \forall k\in\mathbf{Z}.$$
Since $\Phi_a$ not always vanishes,  there exists a $1\leq j_a\leq l$ such that
$$\frac{(2k_a\pi+\theta_{r_a})T_{j_a}}{T}=2\pi n_a$$
for some $n_a\in\mathbf{Z}\backslash\{0\}$. Now we put the $\Phi_j$ whose period are rational commensurable  together and adjust their powers so that the periods are the same.
Let $\sigma_1,\cdots,\sigma_m$ be a rearrangement of $1,\cdots,m$,
$$\Upsilon=(\Phi^{h_1}_{\sigma_1},\cdots,\Phi^{h_m}_{\sigma_m})^\top=(\Upsilon_1^\top,\cdots,\Upsilon_l^\top)^\top$$
such that
$\Upsilon_j(\mathcal{Q}(s)x)=e^{\sqrt{-1}\alpha_j s}\Upsilon_j(x),$
$\alpha_j T_j=2\beta_j\pi$ for some $\beta_j\in\mathbf{Z}\setminus\{0\}$, $1\leq j\leq l$,
where $\Phi_k^{h_k}=(\Phi_{k,1}^{h_k},\cdots,\Phi_{k,q_k}^{h_k})^\top$, $h_k\in\mathbf{N}_+$ for $1\leq k\leq m$.
Then $\Upsilon_j(x^r)=0$ if $j\neq r$.
It follows from $q<n$ and \eqref{3.16} that there must exist a $1\leq j\leq m$ such that the dimension of $\Upsilon_j$
is less than the dimension of $x^{j}$. By Lemma \ref{lemM}, there exists an $x^{j}_0\in M$ such that $\Upsilon_j(x^j_0)=0$.
Then $\Upsilon(x^j_0)=0$ and  $\Phi(x^{j}_0)=0$,
which contradicts \eqref{002.17}, proving the theorem.
\end{proof}

\section{Proof of main results}

Since $C$ is strictly convex, for every $z\in S$, there exist a unique $\zeta\in S^{2n-1}$,
and a real number $r(\zeta)>0$ such that $z=r(\zeta)\zeta$. Since $H\in C^1(\mathbf{R}^{2n},\mathbf{R})$, the implicit function theorem implies that
$r(\zeta)$ is continuously differentiable. Note that
$$Qz=r(Q\zeta)Q\zeta=r(\zeta)Q\zeta\ \ \ \forall z\in S.$$
Then $r(Q\zeta)=r(\zeta)$ for every $\zeta\in S^{2n-1}$.

Consider the function
\begin{eqnarray}\label{2.3}
\mathcal{H}(z)=\bigg\{\begin{array}{ccc}
\frac{1}{q}r\left(\frac{z}{|z|}\right)^{-q}|z|^q,&\ \ {\rm if}\ z\neq0,\\
0,&\ \ {\rm if}\ z=0,
\end{array}
\end{eqnarray}
where $1<q<2$ is a fixed number. Then
\begin{eqnarray}\label{2.4}
\mathcal{H}(Qz)=\mathcal{H}(z)\ \ \forall z\in\mathbf{R}^{2n}.
\end{eqnarray}
It is easy to see that
$$S=\{z: \mathcal{H}(z)=\frac{1}{q}\}.$$
Since $H$ and $\mathcal{H}$ are both invariant on $S$, there exists a continuous function
$\kappa(z)>0$ such that
\begin{eqnarray}\label{2.5}
\nabla H(z)=\kappa(z)\nabla\mathcal{H}(z)\ \ \forall z\in S.
\end{eqnarray}
From $H(z)=H(Qz)$ and \eqref{2.4}, one has
$$\nabla H(z)=Q^\top\nabla H(Qz),\ \ \ \nabla \mathcal{H}(z)=Q^\top\nabla \mathcal{H}(Qz).$$
Then by \eqref{2.5},
$$\nabla H(Qz)=\kappa(z)\nabla \mathcal{H}(Qz).$$
Thus we have
$$\kappa(Qz)=\kappa(z).$$

The following result shows that $H$ and $\mathcal{H}$ have the same $Q$-rotating orbits on $S$.
\begin{lemma}
Assume $\tilde{z}(t)$ is a $Q$-rotating periodic solution of Hamiltonian $\mathcal{H}$ on $S$, then
$$z(t)=\tilde{z}(s(t))$$
is a $Q$-rotating periodic solution of Hamiltonian $H$ on $S$, where $s$ is the solution of the equation
\begin{eqnarray}\label{2.16}
s'=\kappa(\tilde{z}(s)).
\end{eqnarray}
\end{lemma}
\begin{proof}
Assume $\tilde{T}>0$ is a constant, such that
$$\tilde{z}(s+\tilde{T})=Q\tilde{z}(s)\ \ \forall s\in\mathbf{R}.$$
By \eqref{2.16}, one has
\begin{eqnarray}\label{2.7}
h(s)=\int_a^s\frac{\mathrm{d}\tau}{\kappa(\tilde{z}(\tau))}=t+t_0,
\end{eqnarray}
for some constant $a$ and $t_0$. Then $z(t)=\tilde{z}(h^{-1}(t+t_0))$, and
\begin{align*}
z'(t)&=\tilde{z}'(h^{-1}(t+t_0))\frac{1}{h'(h^{-1}(t+t_0))}\\
&=J\nabla \mathcal{H}(\tilde{z}(h^{-1}(t+t_0)))\kappa(\tilde{z}(h^{-1}(t+t_0)))\\
&=J\kappa(z(t))\nabla \mathcal{H}(z(t))\\
&=J\nabla H(z(t)).
\end{align*}
Since $\kappa(\tilde{z}(t+\tilde{T}))=\kappa(\tilde{z}(t))$, we know that
\begin{eqnarray}\label{2.8}
\int_s^{s+\tilde{T}}\frac{\mathrm{d}t}{\kappa(\tilde{z}(t))}=T\ \ \forall s\in\mathbf{R},
\end{eqnarray}
where $T>0$ is a constant.
Since $\kappa(\tilde{z}(t))>0$, by \eqref{2.7} and \eqref{2.8}, one has
$$s(t+T)=s(t)+\tilde{T}\ \ \forall t\in\mathbf{R}.$$
Then for every $t\in\mathbf{R}$,
\begin{eqnarray*}
z(t+T)&=&\tilde{z}(s(t+T))=\tilde{z}(s(t)+\tilde{T})\\
&=&Q\tilde{z}(s(t))=Qz(t).
\end{eqnarray*}
\end{proof}

Now we only need to study $Q$-rotating periodic orbits of Hamiltonian $\mathcal{H}$ on $S$, and still write $\mathcal{H}$ as $H$. By assumption, we have
\begin{align}
\frac{1}{qR^q}|z|^q\leq H(z)\leq\frac{1}{qr^q}|z|^q,\ \  z\in\mathbf{R}^{2n}.
\end{align}
Let $H^*$ be the Legendre-Fenchel transform of $H$, that is
$$H^*(y)=\sup\{\langle z,y\rangle-H(z);z\in\mathbf{R}^{2n}\}.$$
By the expression of $H$, $H^*$ is finite everywhere, $H^*(0)=0$, $H^*\geq0$, and
$$H^*(Qy)=H^*(y)\ \ \ \forall y\in\mathbf{R}^{2n}.$$
Moreover, we have
$H^*(sy)=s^pH^*(y)$ with $p=\frac{q}{q-1}>2$ and
\begin{align}\label{3.25}
\frac{1}{p}r^p|y|^p\leq H^*(y)\leq\frac{1}{p}R^p|y|^p,\ \  y\in\mathbf{R}^{2n}.
\end{align}

It is easy to see that $H$ is strictly convex and $\nabla H$ is strong monotone, see \cite{Struwe}. The following result is classical.

\begin{lemma}\label{lemma1}(see \cite{Struwe}).
Assume $G\in C^1(\mathbf{R}^{2n},\mathbf{R})$ is strictly convex and $\nabla G$ is strongly monotone,
that is, there exists a non-decreasing function $\alpha:[0,\infty)\rightarrow[0,\infty)$ vanishing only at $0$ and satisfying
$\lim\limits_{r\rightarrow\infty}\alpha(r)=\infty$, such that
$$\langle v-w,\nabla G(v)-\nabla G(w)\rangle\geq\alpha(\|v-w\|)\|v-w\|,$$
for all $v,w\in \mathbf{R}^{2n}$. Then $G^*\in C^1(\mathbf{R}^{2n},\mathbf{R})$ and $\nabla G^*(v^*)=v$ for any $v^*=\nabla G(v)$.
\end{lemma}

If $z(t)$ is a solution of system \eqref{1.1}, denote by $y=-Jz'$, then
\begin{eqnarray}
y=\nabla H(z).
\end{eqnarray}
By Lemma \ref{lemma1}, it is equivalent to $z=\nabla H^*(y)$.
Denote
$$\mathcal{B}=\{y\in L^p([0,T],\mathbf{R}^{2n}):y(t+T)=Qy(t),\ \lim\limits_{\tau\rightarrow\infty}\frac{1}{\tau}\int_0^\tau y(t)\mathrm{d}t=0\},$$
$$\mathcal{B}_1=\{y\in W^{1,p}([0,T],\mathbf{R}^{2n}): y(t+T)=Qy(t)\ \ \forall t\in\mathbf{R}\}.$$
Clearly, $\mathcal{B}$ and $\mathcal{B}_1$ are $\mathcal{Q}(s)$-invariant. Let $\mathcal{P}:\mathbf{R}^{2n}\rightarrow\ker(I_{2n}-Q)$ be the orthogonal projection, and $L_\mathcal{P}=(I_{2n}-Q)\mid_{{\rm Im}(I_{2n}-Q)}$.
We introduce the operator $K:\mathcal{B}\rightarrow\mathcal{B}_1$,
$$(Ky)(t)=\int_0^tJy(s)\mathrm{d}s-L_\mathcal{P}^{-1}\int_0^TJy(s)\mathrm{d}s.$$
Now consider the system
\begin{eqnarray}\label{2.12}
z=Ky+z_0,
\end{eqnarray}
\begin{eqnarray}\label{2.13}
z=\nabla H^*(y),
\end{eqnarray}
for some $z_0\in\ker(I_{2n}-Q)$. Since $W^{1,p}\hookrightarrow C^0$, it is easy to see that $z(t)$ is a $(Q,T)$-rotating periodic solution of system \eqref{1.1}, if and only if it is a solution of \eqref{2.12} and \eqref{2.13}.
System \eqref{2.12} and \eqref{2.13} is equivalent to the following system:
\begin{eqnarray}\label{2.15}
\int_0^T\left\langle(\nabla H^*(y)-Ky)(t),\psi(t)\right\rangle\mathrm{d}t=0\ \ \forall \psi\in \mathcal{B}.
\end{eqnarray}
It is easy to see \eqref{2.15} is the Euler-Lagrange equation of the following functional $E$ on $\mathcal{B}$,
$$E(y)=\int_0^T\left(H^*(y)-\frac{1}{2}\langle y,Ky\rangle\right)\mathrm{d}t.$$

We have to prove
\begin{lemma}
$E$ satisfies Palais-Smale condition.
\end{lemma}
\begin{proof}
It is easy to see that $E$ is coercive on $\mathcal{B}$, that is
$$\lim\limits_{\|y\|\rightarrow\infty}E(y)=\infty.$$
Assume $\{y_k\}$ is a (P.-S.) sequence of $E$. Then there exists a constant $c>0$ such that
$$\|y_m\|\leq c\ \ \ \forall m\in\mathbf{N}_+,$$
which implies that there is a weakly convergent subsequence, still denoted by $\{y_m\}$, such that
$y_m\rightharpoonup y$ in $\mathcal{B}$ and $Ky_m\rightarrow Ky$ strongly in $\mathcal{B}_2$, where
$$\mathcal{B}_2=\{y\in L^q([0,T],\mathbf{R}^{2n}): y(t+T)=Qy(t) \ \forall t\in\mathbf{R}\}.$$
Since $H^*$ is strictly convex and homogeneous on rays of degree $p>1$, $\nabla H^*$ is strong monotone.
Hence we have
\begin{align*}
o(1)\|y_m-y\|&=\langle y_m-y,E'(y_m)-E'(y)\rangle\\
&\geq\int_0^T\left(\langle y_m-y,\nabla H^*(y_m)-\nabla H^*(y)\rangle-\langle y_m-y,Ky_m-Ky\rangle\right)\mathrm{d}t\\
&\geq\alpha(\|y_m-y\|)\|y_m-y\|-o(1),
\end{align*}
where $o(1)\rightarrow0$ as $m\rightarrow\infty$, and $\alpha$ is a nonnegative non-decreasing function
that vanishes only at $0$. Hence we obtain that $y_m\rightarrow y$ strongly in $\mathcal{B}$.
\end{proof}

Consider the following group action on $\mathcal{B}$:
$$\mathcal{Q}(s)y=y_s(t)=y(t+s),\ {\rm for}\ s\in\mathbf{R}.$$
We have
\begin{lemma}
$E$ is $\mathcal{Q}(s)$-invariant, that is, for each $y\in\mathcal{B}$
$$E(\mathcal{Q}(s)y)=E(y)\ \ \forall s\in\mathbf{R}.$$
\end{lemma}
\begin{proof}
By a simple calculation, we have
\begin{align*}
K(\mathcal{Q}(s)y)(t)&=\int_0^tJy(\tau+s)\mathrm{d}\tau-L_\mathcal{P}^{-1}\int_0^TJy(\tau+s)\mathrm{d}\tau\\
&=\int_s^{t+s}Jy(\tau)\mathrm{d}\tau-L_\mathcal{P}^{-1}\int_s^{T+s}Jy(\tau)\mathrm{d}\tau\\
&=\int_0^{t+s}Jy(\tau)\mathrm{d}\tau-L_\mathcal{P}^{-1}\int_0^{T}Jy(\tau)\mathrm{d}\tau\\
&=(Ky)(t+s).
\end{align*}
Notice that $H^*(y(t))$ and $\langle y(t),(Ky)(t)\rangle$ are both $T$ periodic, which implies
\begin{align*}
E(\mathcal{Q}(s)y)&=\int_0^T\left(H^*(y(t+s))-\frac{1}{2}\big\langle y(t+s),(Ky)(t+s)\big\rangle\right)\mathrm{d}t\\
&=\int_s^{T+s}\left(H^*(y(t))-\frac{1}{2}\big\langle y(t),(Ky)(t)\big\rangle\right)\mathrm{d}t\\
&=E(y).
\end{align*}
\end{proof}

Let $$m=\inf\{E(y):y\in\mathcal{B}\},$$
$$\mathcal{B}'=\{y\in \mathcal{B}: {\rm there\ exists\ a}\ 0<T_1<T, {\rm\ such\ that}\ y(t+T_1)=Qy(t)\},$$
$$m^*=\inf\{E(y):y\in \mathcal{B}'\}.$$

We have the following estimates.
\begin{lemma}\label{lem3.5}
There exist $y\in\mathcal{B}$ and $y^*\in\mathcal{B}'$,
such that $E(y)=m,\ E(y^*)=m^*,$ and
$$m\leq2^{\frac{p}{p-2}}m^*<0.$$
\end{lemma}
\begin{proof}
Notice that $E$ is weakly lower semi-continuous and coercive on $\mathcal{B}$, $\mathcal{B}$ and $\mathcal{B}'$
are weakly closed. Then there exist $y\in\mathcal{B}$ and $y^*\in\mathcal{B}'$,
such that $E(y)=m,\ E(y^*)=m^*.$
Assume
\begin{align}\label{03.30}
y^*(t+T_1)=Qy^*(t)\ \ \forall t\in\mathbf{R},
\end{align}
 where $0<T_1<T$ is the minimum positive $Q$-rotating period of $y^*$. Since $y^*$ is the minimum of $E$ in $\mathcal{B}'$, one has
$$\langle y^*,E'(y^*)\rangle=p\int_0^TH^*(y^*)\mathrm{d}t-\int_0^T\langle y^*,Ky^*\rangle\mathrm{d}t=0.$$
It follows that
$$m^*=\left(\frac{1}{p}-\frac{1}{2}\right)\int_0^T\langle y^*,Ky^*\rangle\mathrm{d}t<0.$$
Let $\bar{y}(t)=y^*(\frac{T_1 t}{T})$. Then one has
$$\bar{y}(t+T)=y^*(\frac{T_1 t}{T}+T_1)=Qy^*(\frac{T_1 t}{T})=Q\bar{y}(t),$$
and
\begin{align*}
\lim\limits_{\tau\rightarrow\infty}\frac{1}{\tau}\int_0^\tau \bar{y}(t)\mathrm{d}t&=\lim\limits_{\tau\rightarrow\infty}\frac{1}{\tau}\int_0^\tau y^*\left(\frac{T_1 t}{T}\right)\mathrm{d}t\\
&=\lim\limits_{\tau\rightarrow\infty}\frac{T}{T_1\tau}\int_0^{\frac{T_1\tau}{T}}y^*(t)\mathrm{d}t\\
&=0.
\end{align*}
Then $\bar{y}\in\mathcal{B}$, and
\begin{align}\label{03.31}
m\leq\inf_{s>0}E(s\bar{y})=\inf_{s>0}\left(s^p\int_0^TH^*(\bar{y})\mathrm{d}t-\frac{s^2}{2}\int_0^T\langle \bar{y},K\bar{y}\rangle\mathrm{d}t\right).
\end{align}
By a simple calculation, we have
\begin{align}\label{3.30}
\int_0^TH^*(\bar{y}(t))\mathrm{d}t=\frac{T}{T_1}\int_0^{T_1}H^*(y^*(t))\mathrm{d}t.
\end{align}
Consider the function $H^*(y^*(t))$. Then
\begin{align*}
H^*(y^*(t+T_1))=H^*(y^*(t+T))=H^*(y^*(t)).
\end{align*}
If $k_1T_1+k_2T\neq0$ for any $k_1,k_2\in\mathbf{Z}$, one has
\begin{align}\label{3.31}
H^*(y^*(t))=c\ \ \ \forall t\in\mathbf{R},
\end{align}
for some constant $c$.
If there exist $k_1,k_2\in\mathbf{Z}$ such that $k_1T_1+k_2T=0$, without loss of generality, we assume that
$T_1=m_1\tilde{T}$, $T=m_2\tilde{T}$, where $m_1, m_2\in\mathbf{N}$ and $\tilde{T}$ is the minimum positive period of $H^*(y^*(t))$.
Then we have
\begin{align*}
\int_0^{T_1}H^*(y^*(t))\mathrm{d}t=m_1\int_0^{\tilde{T}}H^*(y^*(t))\mathrm{d}t,
\end{align*}
\begin{align*}
\int_0^{T}H^*(y^*(t))\mathrm{d}t=m_2\int_0^{\tilde{T}}H^*(y^*(t))\mathrm{d}t.
\end{align*}
Then, one has
\begin{align}\label{3.32}
\int_0^{T}H^*(y^*(t))\mathrm{d}t=\frac{m_2}{m_1}\int_0^{T_1}H^*(y^*(t))\mathrm{d}t=\frac{T}{T_1}\int_0^{T_1}H^*(y^*(t))\mathrm{d}t.
\end{align}
By \eqref{3.30}-\eqref{3.32}, we have
\begin{align}\label{3.33}
\int_0^TH^*(\bar{y}(t))\mathrm{d}t=\int_0^TH^*(y^*(t))\mathrm{d}t.
\end{align}
Since $y^*(t+T_1)=Qy^*(t)$ and $y^*(t+T)=Qy^*(t)$, one has
$$y^*(t+T-T_1)=y^*(t).$$
It follows from $y^*\in\mathcal{B}$ that
\begin{align*}
\int_{T_1}^Ty^*(t)\mathrm{d}t=\lim\limits_{\tau\rightarrow\infty}\frac{1}{\tau}\int_0^\tau y^*(t)\mathrm{d}t=0,
\end{align*}
which implies that
\begin{align}
\int_0^Ty^*(t)\mathrm{d}t=\int_0^{T_1}y^*(t)\mathrm{d}t.
\end{align}
Then we have
\begin{align*}
\int_0^T\langle \bar{y}(t), K\bar{y}(t)\rangle\mathrm{d}t&=\int_0^T\left\langle \bar{y}(t), \int_0^tJ\bar{y}(s)\mathrm{d}s-L_\mathcal{P}^{-1}\int_0^TJ\bar{y}(s)\mathrm{d}s\right\rangle\mathrm{d}t\\
&=\frac{T}{T_1}\int_0^T\left\langle y^*(\frac{T_1t}{T}), \int_0^{\frac{T_1t}{T}}Jy^*(s)\mathrm{d}s-L_\mathcal{P}^{-1}\int_0^{T_1}Jy^*(s)\mathrm{d}s\right\rangle\mathrm{d}t\\
&=\frac{T^2}{T_1^2}\int_0^{T_1}\langle y^*(t), Ky^*(t)\rangle\mathrm{d}t.
\end{align*}
Hence analogous to \eqref{3.33}, we have
\begin{align}\label{3.37}
\int_0^T\langle \bar{y}(t), K\bar{y}(t)\rangle\mathrm{d}t=\frac{T}{T_1}\int_0^T\langle y^*(t), Ky^*(t)\rangle\mathrm{d}t.
\end{align}
We claim that $T\geq 2T_1$, otherwise $T_1-(T-T_1)>0$ and
$$y^*(t+(T_1-(T-T_1)))=Qy^*(t)\ \ \forall t\in\mathbf{R},$$
which is in contradiction with the minimum of $T_1$.
By \eqref{03.31}, \eqref{3.33} and \eqref{3.37}, we have
\begin{align*}
m&\leq\inf_{s>0}\left(\frac{s^p}{p}-\frac{s^2T}{2T_1}\right)\int_0^T\langle y^*,Ky^*\rangle\mathrm{d}t\\
&=\left(\frac{T}{T_1}\right)^{\frac{p}{p-2}}\left(\frac{1}{p}-\frac{1}{2}\right)\int_0^T\langle y^*,Ky^*\rangle\mathrm{d}t\\
&\leq2^{\frac{p}{p-2}}m^*.
\end{align*}
\end{proof}

Let
$$\Gamma=\left\{z\in\mathcal{B}:\|z\|=1\right\},\ \ b=\sup\left\{\int_0^T\langle z,Kz\rangle\mathrm{d}t:z\in\Gamma\right\}.$$
Then we have
\begin{lemma}
$m\geq c_0r^{\frac{2p}{2-p}}$ with $c_0=\left(\frac{1}{p}-\frac{1}{2}\right)b^{\frac{p}{p-2}}$.
\end{lemma}
\begin{proof}
Assume $E(y)=m$ with $y\in\mathcal{B}$. Then we have
\begin{align}\label{3.38}
\langle y,E'(y)\rangle=p\int_0^TH^*(y)\mathrm{d}t-\int_0^T\langle y,Ky\rangle\mathrm{d}t=0,
\end{align}
which implies
$$m=\left(\frac{1}{p}-\frac{1}{2}\right)\int_0^T\langle y,Ky\rangle\mathrm{d}t=\left(1-\frac{p}{2}\right)\int_0^TH^*(y)\mathrm{d}t.$$
Assume $y=\lambda z$ for some $z\in\Gamma$, $\lambda\geq0$. It follows from \eqref{3.25} and \eqref{3.38} that
\begin{align*}
\lambda^{2-p}b\geq\lambda^{2-p}\int_0^T\langle z,Kz\rangle\mathrm{d}t=p\int_0^TH^*(z)\mathrm{d}t\geq r^p.
\end{align*}
Since $p>2$, we have
$$\lambda\leq r^{\frac{p}{2-p}}b^{\frac{1}{p-2}}.$$
Hence
\begin{align*}
m&=\left(\frac{1}{p}-\frac{1}{2}\right)\lambda^2\int_0^T\langle z,Kz\rangle\mathrm{d}t\geq\left(\frac{1}{p}-\frac{1}{2}\right)\lambda^2b
\geq c_0r^{\frac{2p}{2-p}}.
\end{align*}
\end{proof}

\begin{lemma}\label{lem3.7}
Assume $\int_0^T\langle \bar{z},K\bar{z}\rangle\mathrm{d}t=b$ with $\bar{z}\in\Gamma$. Then $\bar{z}\in M$, where $M$ is defined in Theorem \ref{th2}. Moreover,
for each $z\in M$, one has
\begin{align*}
\frac{T^2}{\tilde{\theta}_n}\leq\int_0^T\langle z,Kz\rangle\mathrm{d}t\leq\frac{T^2}{\tilde{\theta}_1}=b.
\end{align*}
\end{lemma}
\begin{proof}
Let $V(z)=|z|^p$. Then there exist some $\lambda\in\mathbf{R}$ and $\xi\in\ker(I_{2n}-Q)$ such that
\begin{align}\label{3.41}
K\bar{z}=\lambda\nabla V(\bar{z})+\xi.
\end{align}
As in \eqref{3.37}, $T$ is the minimum positive $Q$-rotating period of $\bar{z}$.
Equation \eqref{3.41} is equivalent to the system
\begin{eqnarray}\label{3.42}
\bigg\{\begin{array}{ccc}
u=K\bar{z}-\xi,\\
u=\lambda\nabla V(\bar{z}).
\end{array}
\end{eqnarray}
By the duality principle, $u$ satisfies the system
\begin{align}\label{3.43}
u'=cJ\nabla w(u),
\end{align}
where $c$ is a constant and $w(u)=|u|^q$.
Since $w(u(t))$ is constant for all $t\in\mathbf{R}$, we see $|u(t)|$ is constant for $t\in\mathbf{R}$. Then \eqref{3.43}
has the form
\begin{align}
u'=\tilde{c}J u,
\end{align}
for some constant $\tilde{c}\in\mathbf{R}$. Since $T$ is the minimum positive $Q$-rotating period, we have that
$\tilde{c}=\frac{\tilde{\theta}_i}{T}$ for some $1\leq i\leq n$, and $u(t)$ is in the invariant subspace corresponding to eigenvalues $e^{\pm \sqrt{-1}\tilde{\theta}_i}$ of $Q$.
That is,
$$u(t)=e^{\frac{\tilde{\theta}_i}{T}Jt}u(0)=Q(t)u(0)$$
with some $u(0)$ in the invariant subspace corresponding to eigenvalues $e^{\pm \sqrt{-1}\tilde{\theta}_i}$ of $Q$. Then
$$\bar{z}(t)=-Ju'(t)=e^{\frac{\tilde{\theta}_i}{T}Jt}\bar{z}(0)=Q(t)\bar{z}(0),$$
for some $1\leq i\leq n$, where $\bar{z}(0)$ is in the invariant subspace corresponding to eigenvalues $e^{\pm \sqrt{-1}\tilde{\theta}_i}$ of $Q$ and $|\bar{z}(0)|=1$.
Through a simple calculation, we see that
$\int_0^T\langle z,Kz\rangle\mathrm{d}t$ attains the maximum, when $\tilde{\theta}_i=\tilde{\theta}_1$.
Moreover, for each $z\in M$, we have
$$\frac{T^2}{\tilde{\theta}_n}\leq\int_0^T\langle z,Kz\rangle\mathrm{d}t\leq\frac{T^2}{\tilde{\theta}_1}=b.$$
\end{proof}

Now we give the proof of Theorem \ref{th1}.
\begin{proof}[Proof of Theorem \ref{th1}]
For each $z\in M$, choose $\lambda=\lambda(z)>0$ such that
$$p\lambda^p\int_0^TH^*(z)\mathrm{d}t-\lambda^2\int_0^T\langle z,Kz\rangle\mathrm{d}t=0.$$
Consider the continuous map $\rho(z)=\lambda(z)z$ on $M$. It is easy to see that
$$\rho(\mathcal{Q}(s)z)=\mathcal{Q}(s)\rho(z)\ \ \ \forall s\in\mathbf{R}.$$
Then by Theorem \ref{th2} one has
$$\ind\rho(M)\geq\ind M=n.$$
By \eqref{3.25} and Lemma \ref{lem3.7}, for each $z\in M$ we obtain
\begin{align*}
\lambda^{2-p}b\leq\frac{\lambda^{2-p}\tilde{\theta}_n}{\tilde{\theta}_1}\int_0^T\langle z,Kz\rangle\mathrm{d}t=\frac{p\tilde{\theta}_n}{\tilde{\theta}_1}\int_0^TH^*(z)\mathrm{d}t
\leq\frac{\tilde{\theta}_nTR^p}{\tilde{\theta}_1},
\end{align*}
which implies that
\begin{align*}
\lambda\geq\left(\frac{\tilde{\theta}_1}{\tilde{\theta}_n}\right)^{\frac{1}{p-2}}T^{\frac{1}{2-p}}R^{\frac{p}{2-p}}b^{\frac{1}{p-2}}.
\end{align*}
Hence
\begin{align*}
\sup_{y\in\rho(M)}E(y)&=\sup_{z\in M}\int_0^T\left(\lambda^p H^*(z)-\frac{\lambda^2}{2}\langle z,Kz\rangle\right)\mathrm{d}t\\
&=\sup_{z\in M}\left(\frac{1}{p}-\frac{1}{2}\right)\lambda^2\int_0^T\langle z,Kz\rangle\mathrm{d}t\\
&\leq\left(\frac{1}{p}-\frac{1}{2}\right)\left(\frac{\tilde{\theta}_1}{\tilde{\theta}_n}\right)^{\frac{2}{p-2}}T^{\frac{2}{2-p}}R^{\frac{2p}{2-p}}b^{\frac{p}{p-2}}\\
&=c_0\left(\bigg(\frac{\tilde{\theta}_1}{\tilde{\theta}_n}\bigg)^{-\frac{1}{p}}T^{\frac{1}{p}}R\right)^{\frac{2p}{2-p}}.
\end{align*}
Since $R<\sqrt{2}r$, there exists a sufficiently large $p$ such that
$$\bigg(\frac{\tilde{\theta}_1}{\tilde{\theta}_n}\bigg)^{-\frac{1}{p}}T^{\frac{1}{p}}R<\sqrt{2}r.$$
It follows that
$$\sup_{y\in\rho(M)}E(y)<2^{\frac{p}{2-p}}c_0r^{\frac{2p}{2-p}}\leq m^*.$$
For $1\leq j\leq n$, by \eqref{2.14} one has
$$c_j\leq\sup_{y\in\rho(M)}E(y)<2^{\frac{p}{2-p}}c_0r^{\frac{2p}{2-p}}<0.$$
It follows from $\fix\{\mathcal{Q}(s)\}=\{0\}$ in $\mathcal{B}$ and $E(0)=0$ that
$$K_{c_j}\cap\fix\{\mathcal{Q}(s)\}=\emptyset,\ \ \ 1\leq j\leq n,$$
where $c_j$ is defined by \eqref{02.15}.
By Theorem \ref{th3.1} and Lemma \ref{p2}, $E(y)$ has at least $n$ $\mathcal{Q}(s)$-critical orbits
$\mathcal{Q}(s)y_j\in K_{c_j}$ for $1\leq j\leq n$.
Hence $z_j=\nabla H^*(y_j)$ is a $(Q,T)$-rotating periodic solution of \eqref{1.1}.
Since $E(y_j)<m^*$, $T$ is the minimum positive $Q$-rotating period of $y_j$. Clearly, $T$ is also the minimum positive $Q$-rotating period of $z_j$. Assume $d_j= H(z_j)$ and let
$$w_j(t)=(qd_j)^{-\frac{1}{q}}z_j((qd_j)^{\frac{2-q}{q}}t).$$
Then $w_j(t)$ is a $(Q,(qd_j)^{\frac{q-2}{q}}T)$-rotating periodic solution of \eqref{1.1} on $H^{-1}(\frac{1}{q})$.
If $w_j$ and $w_k$ describe the same orbit, there exists some $s_0\in\mathbf{R}$ such that
$w_k=\mathcal{Q}(s_0)w_j$ which implies $d_k=d_j$. Hence $z_j$ and $z_k$ also describe the same orbit, that is,
$z_k=\mathcal{Q}(s_1)z_j$ for some $s_1\in\mathbf{R}$. Then
$$y_k=-Jz_k'=-\mathcal{Q}(s_1)Jz_j'=\mathcal{Q}(s_1)y_j.$$
It follows that $y_j$ and $y_k$ describe the same orbit.
Thus, there exist at least $n$ $Q$-rotating periodic orbits of \eqref{1.1} on the energy surface $S$, proving the theorem.
\end{proof}
\begin{remark}
For the case $Q=I_{2n}$, the $\mathcal{Q}(s)$ index and $S^1$ index are consistent. In applying Theorem \ref{th3.1}, if $c_i\neq c_j$ for all $1\leq i, j\leq n$ and $i\neq j$, then there are exactly $n$ periodic orbits; however, if there exist $i\neq j$ such that $c_j=c_j$, then there are infinitely many periodic orbits.
\end{remark}

\baselineskip 9pt \renewcommand{\baselinestretch}{1.08}

\end{document}